\documentclass{l4dc2025}

\title[GWShapeCon]{Formation Shape Control using the Gromov-Wasserstein Metric}
\usepackage{times}
\usepackage{defs}
\usepackage{stmaryrd}    
\usepackage{graphicx,wrapfig,lipsum}
\usepackage{epstopdf}
\usepackage{subcaption}
\usepackage{hyperref}
\usepackage{float}
\usepackage{array, multirow, url}
\usepackage{caption}
\usepackage{tabularray}

\coltauthor{%
 \Name{Haruto Nakashima} \Email{nakashima.haruto.72v@st.kyoto-u.ac.jp}\\
 \Name{Siddhartha Ganguly} \Email{ganguly.siddhartha.7p@kyoto-u.ac.jp}\\
  \Name{Kohei Morimoto} \Email{kohei.morimoto.73r@st.kyoto-u.ac.jp}\\
  \Name{Kenji Kashima} \Email{kk@i.kyoto-u.ac.jp}\\
 \addr Graduate School of Informatics, Kyoto University, Kyoto, Japan
}

\begin{document}

\maketitle

\begin{abstract}%
This article introduces a formation shape control algorithm, in the optimal control framework, for steering an initial population of agents to a desired configuration via employing the Gromov-Wasserstein distance. The underlying dynamical system is assumed to be a constrained linear system and the objective function is a sum of quadratic control-dependent stage cost and a Gromov-Wasserstein terminal cost. The inclusion of the Gromov-Wasserstein cost transforms the resulting optimal control problem into a well-known NP-hard problem, making it both numerically demanding and difficult to solve with high accuracy. Towards that end, we employ a recent semi-definite relaxation-driven technique to tackle the Gromov-Wasserstein distance. A numerical example is provided to illustrate our results.
\end{abstract}


\section{Introduction}


The Gromov-Wasserstein (GW) distance (\cite{ref:FM:GW:2011}) is an important and popular optimal transport (OT) distance which is often employed to compare distributions whose underlying spaces are \emph{incomparable}, as opposed to the classical OT distances, for example, the Wasserstein distance. The GW distance recently gained ample amount of attention in various scientific communities ranging from machine learning \cite{ref:GWML:graph_partitioning} to mathematical finance \cite{ref:GWML:math_finance} and system biology \cite{ref:GWML:cell_biology}. The GW distance is invariant under isometric transformations, such as translation, rotation, and reflection, and thus it is especially useful in applications where the underlying metric geometry, for example, shape or formation, needs to be preserved; see \cite{ref:GWML:allignment}, \cite{ref:GWML:shape_matching}, \cite{ref:GWML:spectral_shape_matching}, \cite{ref:GWML:graph_matching}.

This article establishes a formation/shape control algorithm in the minimum energy optimal control framework by employing the GW distance and leveraging its invariance and metric properties. We outline our primary motivation: 
\vspace{1mm}
\begin{enumerate}[leftmargin=*, widest=b, align=left]

\item \label{motivation:I} Traditional control approaches, such as leader-follower or distance-based methods \cite{ref:DISFORMCTRL:2008, ref:DISFORMCTRL:2016}, rely on specifying and pre-assigning precise target positions or inter-agent distances within a fixed reference frame. While these methods work well for simple, static formations, they are often rigid and struggle to adapt to dynamic environments or objectives that do not require strict positional accuracy. This rigidity makes them less effective in scenarios where the overall shape or configuration of the agents is essential, while the distances between each pair of agents are of secondary importance. In contrast, the GW distance allows for comparing distributions based on their relative geometries, making it invariant to translations, reflections, and rotations.

\item \label{motivation:II} The invariance properties of the GW distance can indeed be useful in the context of controlling the shape of given formation, where, for example, for an ensemble of homogeneous agents (for example, a multi-robotic swarm), the goal may be to create a specific shape, such as a circle, line, or more complex configuration, without concern for the precise orientation or placement of the formation. For example, consider a scenario where a group of drones is required to spread out into a circular formation for optimal coverage in a search-and-rescue mission. Using conventional control techniques, each drone would need to follow precise, fixed coordinates or inter-agent distances, which can be cumbersome and inflexible if the environment changes or if the entire formation needs to rotate or translate. In the existing literature on formation control of ensemble systems such as \cite{ref:FORMCTRL:Alejandro:2017, ref:FORMCTRL:Chen:2017}, only configurations with fixed orientations have been considered, while the optimization over distance-preserving transformations, such as rotation and translation, has not been addressed. On the other hand, in our formulation, incorporating GW distance as the objective function in formation control adds flexibility because the distance itself accounts only for the shape similarity between two distributions rather than for absolute positions and orientations. This means that agents can achieve the desired formation shape with minimal dependency on external coordinate frames or absolute positioning. This can be viewed as a simultaneous optimization of the dynamical steering and the configuration of agents. \\
\end{enumerate}



With this motivation, we introduce the GW distance in the context of the multi-agent formation control problem of linear discrete-time dynamical systems from a constrained optimal control perspective. Our approach treats the formation control problem as a \emph{controlled optimal transport} problem which offers us a systematic way to compare and align multi-agent configurations by utilizing the inherent geometric structure of the agent formations. The GW-optimization problem (see \eqref{eq:disc_GW_dist} ahead) is nonconvex and presents significant computational challenges. To address this, we employ a semidefinite programming (SDP) relaxation of the GW-distance, developed in \cite{ref:JC:BTN:SHK:YSS:SDP-relaxation:2024}, which enables an accurate and computationally tractable approach for solving the optimal transport problem. Furthermore, under certain conditions on the solutions of the SDP, we can achieve \emph{globally optimal solutions} to the original GW problem \eqref{eq:disc_GW_dist} unlike conditional gradient based methods \cite{ref:GWML:ConditionedGD} that searches for locally optimal solutions or entropy regularization-based techniques \cite{ref:GWML:Entropic_GW} which does not solve the original problem and instead solve an approximate problem, minimizing a different regularized objective. 

\par The remainder of the paper is organized as follows: \S\ref{sec:GW_description} introduces the concept of GW distance, \S\ref{sec:problem_formulation} formulates the problem as a minimum-energy optimal control problem. In Section \S\ref{sec:main_results}, we present our primary findings, including the SDP relaxation-based algorithm. We also provide several results that ensure the existence of optimizers for specific optimal control problems. We demonstrate our results through a numerical example in \S\ref{sec:numerics}. Our findings and results suggest that there is ample scope for improvement, both theoretically and numerically, with several potential directions for future research. These possibilities are discussed in detail in \S\ref{sec:conclusion}.

\section{Background on the Gromov-Wasserstein distance}\label{sec:GW_description}
We fix some notations first. For us \(\mathbb{N}\) is the set of natural numbers and \(\N \Let \mathbb{N}\setminus \{0\} = \aset{1, 2, \ldots}\) denote the set of positive integers. For any \(n \in \N\), we define the finite set of positive integers up to \(n\) by \(\intbra{n} \Let \aset{1, \ldots, n}\). Given \(d \in \N\), let \(\Rbb^d_{\succ 0}\) and \(\Rbb^d_{\succeq 0}\) denote the sets of positive and nonnegative vectors in the Euclidean space \(\Rbb^d\), respectively. For any \(x \in \Rbb^{d}\) and for any positive definite matrix \(M \in \Rbb^{d \times d}\), we denote the standard quadratic form \(\inprod{x}{Mx}\) by \(\norm{x}^2_{M}\); when \(M\) is not necessarily positive (semi) definite we employ \(\inprod{x}{Mx} \Let x^{\top}Mx\). Let \(d_1,d_2 \in \N\) and consider the vector space of all \(d_1 \times d_2\) matrices with real entries; we denote a inner product of such matrices by \((A_1,A_2) \mapsto \inprod{A_1}{A_2} \Let \text{tr}(A_2^{\top}A_1)\). We denote the standard probability simplex by \(\simplex{d} := \aset{\zeta \in \Rbb^d_{\succeq 0} \mid \sum_{i=1}^d \zeta_i = 1}\), and \(\bm{1}_N \in \Rbb^N\) represents the vector with all entries equal to \(1\). The standard Euclidean basis vector whose \(i\)-th entry is \(1\) is denoted by \(e_i\).

Let \(\measX(\cdot)\) and \(\measY(\cdot)\) be two given probability density functions specified on metric spaces \(\metspaceX\) and \(\metspaceY\) with respective metrics \(\costX: \metspaceX \times \metspaceX \lra \Rbb\) and \(\costY:\metspaceY \times \metspaceY \lra \Rbb\). Define the set of probability measures with marginals \(\measX(\cdot)\) and \(\measY(\cdot)\) by \(\probmeas(\measX,\measY) \Let \aset[]{\marginal(x,y) \suchthat \int_{\metspaceY} \marginal(x,y)\,\dd y = \measX(x),\, \int_{\metspaceX}\marginal(x,y)\,\dd x = \measY(y)}\). Then for \(p\ge 1\), the GW distance, \(\GW_{p}\) is defined by:
\begin{align}\label{eq:cont_GW_dist}
   \GW_{p} \bigl( \measX, \measY\bigr)^{p} \Let \inf_{\marginal \in \probmeas(\measX,\measY)} \int_{\metspaceX^2 \times \metspaceY^2} \abs{\costX(x,x') - \costY(y,y')}^{p} \dd\marginal(x,y)\dd\marginal(x',y').
\end{align}
We refer the readers to \cite{ref:FM:GW:2011} for an in-depth analysis of the geometric and metric properties of \(\GW(\cdot,\cdot)\). For computational purposes, we focus on the setting of \emph{discrete} metric-measure spaces. To this end, fix natural numbers \(n,m \ge 1\), and let \(\sourcevec \Let (\sourcevec_1,\ldots,\sourcevec_n)^{\top} \in \simplex{n}\) and \(\tarvec\Let (\tarvec_1,\ldots,\tarvec_m)^{\top} \in \simplex{m}\). On \((\metspaceX,\costX)\) and \((\metspaceY,\costY)\) define the discrete probability measures \(\measX \Let \sum_{i=1}^n \sourcevec_i \delta_{x_i}\) and \(\measY \Let \sum_{j=1}^{m} \tarvec_j \delta_{y_j}\). Given two non-negative cost matrices \(\costmatX \in \Rbb^{n \times n}\) and \(\costmatY \in \Rbb^{m\times m}\) associated with the discrete metric-measure spaces \((\costmatX,\sourcevec)\) and \((\costmatY,\tarvec)\), define the set of couplings by \(\probmeas(\sourcevec,\tarvec)\Let \aset[]{\coupmat = [\coupmat_{ij}] \in \Rbb^{n \times m} \suchthat \coupmat \textbf{1}_m = \sourcevec,\, \coupmat^{\top}\textbf{1}_n = \tarvec}.\) 

Consider a smooth \emph{loss} function \(\Rbb \times \Rbb \ni (\xi,\mu)\mapsto \dummysmallfunc(\xi,\mu)\), where \(\dummysmallfunc\bigl(\costmatX_{i,j},\costmatY_{k,l}\bigr)\) measures the difference between the metrics \(\costX(x^{i},x^{i'})\) and \(\costY(y^{j},y^{j'})\) for \(1\le i,i'\le n\) and \(1 \le j,j' \le m\). The values of \(\dummysmallfunc(\cdot,\cdot)\) can be put in a tensor \(\dummybigfunc \in \Rbb^{mn \times mn}\) where \(\dummybigfunc \Let [\dummybigfunc_{iji'j'}] = \dummysmallfunc \bigl(\costmatX_{i,i'},\costmatY_{j,j'}\bigr)\). With these ingredients, the discrete analogue of \eqref{eq:cont_GW_dist} is given by
\begin{align}\label{eq:disc_GW_dist}
\hspace{-3mm}\GW\bigl(\costmatX,\costmatY,\sourcevec,\tarvec\bigr) &\Let \inf_{\coupmat \in \probmeas(\sourcevec,\tarvec)} \dummysmallfunc \bigl( \costmatX_{i,i'},\costmatY_{j,j'}\bigr)\coupmat_{i,j}\coupmat_{i',j'} = \inf_{\coupmat \in \probmeas(\sourcevec,\tarvec)} \sum_{i,j,i',j'}\dummybigfunc_{iji'j'}\coupmat_{i,j}\coupmat_{i',j'}. 
\end{align}
Defining the vectorization of \(\coupmat\) by \(\vectr(\coupmat) \Let [\coupmat_{11},\coupmat_{21},\ldots,\coupmat_{m1},\ldots,\coupmat_{mn}]^{\top} \in \Rbb^{mn}\) we rewrite \eqref{eq:disc_GW_dist} as
\begin{align}\label{eq:disc_GW_dist_vec}
\GW\bigl(\costmatX,\costmatY,\sourcevec,\tarvec\bigr) = \inf_{\coupmat \in \probmeas(\sourcevec,\tarvec)} \inprod{\vectr(P)}{G\vectr(P)},
\end{align}
which is a nonconvex NP-hard optimization problem and computationally intractable in general. We note that if the distributions are assumed to be Gaussian, the expression in \eqref{eq:cont_GW_dist} admits a closed-form solution; see \cite{ref:GaussianGroWass} and \cite{ref:KM:KK:MinDenLin:2024} for an application. While approximation-driven methods \cite{ref:MR:JK:JK:QAPGW-Neurips:2023} and entropic regularization techniques \cite{ref:GWML:Entropic_GW} offer computationally feasible ways to solve \eqref{eq:disc_GW_dist_vec}, they either depend on loss-specific algorithms or lack a certificate of global optimality. In the following section, we adopt an SDP-relaxation approach introduced in \cite{ref:JC:BTN:SHK:YSS:SDP-relaxation:2024} (which provides a certificate of global optimality, contingent on satisfying a specific condition; see \S\ref{sec:problem_formulation}, Remark \ref{rem:on_SDP_relaxation}, Equation \eqref{eq:approx_ration_condition}) to numerically solve \eqref{eq:disc_GW_dist_vec}.
\section{Problem formulation}\label{sec:problem_formulation}
Fix \(N \in \N\) and consider an \(N\)-agent discrete-time control system given by the recursion
\begin{equation}
    \label{eq:system}
    x^{i}_{t+1} = Ax^{i}_t+Bu^{i}_t,\quad \st^i_0=\xz^i\text{ given},\,t \in \Nz,\text{ and }i \in \intbra{N},
\end{equation}
along with the following data:
\begin{enumerate}[label=\textup{(\ref{eq:system}-\alph*)}, leftmargin=*, widest=b, align=left]
\item \label{eq:data:I:system-st-con-dist} for each \(i \in \intbra{N}\), the state and control input vectors at time \(t\) are denoted by \(x_t^i\) and \(u_t^i\), respectively. Moreover, for each \(i \in \intbra{N}\) and \(t \in \aset{0, \ldots, \horizon - 1}\), we have \(x_t^i \in \Mbb \subset \Rbb^{\sysDim}\) and \(u_t^i \in \Ubb \subset \Rbb^{\contDim}\), where \(\Mbb\) and \(\Ubb\) are compact and convex sets with \(0 \in \Rbb^{\sysDim}\) and \(0 \in \Rbb^{\contDim}\) in their respective interiors.
\item \label{eq:data:II:system-final-st} At the final time \(\horizon\), we impose \(\st_{\horizon} \in \admfinst \subset \Rbb^{\sysDim}\), where \(\admfinst\) is a compact and convex set containing \(0 \in \Rbb^{\sysDim}\) in its interior. 
\item  \label{eq:data:III:system-dynamics}  \(A \in \Rbb^{d \times d}\) and \(B \in \Rbb^{d \times m}\) are respectively the system and the control matrices which are, for simplicity, assumed to be the same for all the agents, and the \((A,B)\) pair is stabilizable. 
\end{enumerate}
\vspace{2mm}
For a fixed number of agents \(N \in \N\) and a fixed horizon \(\horizon \in \N\) we consider steering the initial point clouds \(\xz^i\in \Rbb^{\sysDim}\) to the same shape as the target point group \(\xtarg^i \in  \Rbb^{\sysDim}\), for each \(i \in \intbra{N}\), in an energy-efficient manner. Define the set of initial and final states for the agents by \(\xz \Let (\xz^1,\ldots,\xz^N)\) and \(x_d \Let (x_{T}^1,\ldots,x_{T}^N)\), and let \(\xtarg\Let (\xtarg^1,\ldots,\xtarg^N)\). Consider the \(N\)-agent, \(T\)-horizon, minimum energy optimal control problem (OCP): 
\begin{equation}
	\label{eq:finite_horizon_OCP}
	\begin{aligned}
		& \inf_{x_d} \inf_{(u^i_t)_{t=0}^{\horizon-1}}  && \sum_{i=0}^{N} \sum_{t=0}^{\horizon-1} \norm{u_t^i}^2_R + \GW\bigl(x_{d},\xtarg\bigr)\\
		& \quad \quad\sbjto && \begin{cases}
			x_{t+1}^{i} = Ax_t^{i} + B u_t^{i}\quad\text{ for each }t=0,\ldots,\horizon-1,\,\text{and }i \in \intbra{N}, \\
            x_0^i = \xz^i, \text{ and } x^i_T =x^i_{d},\,x_{\horizon}^i \in \admfinst, \\
            x_t^i \in \Mbb,\,u_t^i \in \Ubb \quad\text{ for each }t=0,\ldots,\horizon-1,\,\text{and }i \in \intbra{N},
		\end{cases}
	\end{aligned}
\end{equation}
where \(R \in \Rbb^{\contDim}_{\succ 0}\) is the control weighting matrix, and for compact notation in our context, we denote the GW distance in \eqref{eq:disc_GW_dist_vec} by \(\GW\bigl(x_{d},\xtarg\bigr)\). The OCP \eqref{eq:finite_horizon_OCP} can be compactly written as:
\begin{equation}\label{eq:optimal_navigator}
	\begin{aligned}
		& \inf_{x_{d}}  \rcost(\xz,x_{d}) + \GW\lr{x_{d},\xtarg},
	\end{aligned}
\end{equation}
where quadratic control cost is given by
\begin{equation} \label{eq:finite_horizon_OC}
	\begin{aligned}
	\rcost(\xz,x_d) \Let	 &\inf_{(u^i_t)_{t=0}^{\horizon-1}}  && \sum_{i=0}^{N} \sum_{t=0}^{\horizon-1} \norm{u_t^i}^2_R \\
	&	\,\,\,\, \sbjto &&\begin{cases}
  x_{t+1}^{i} = Ax_t^{i} + B u_t^{i}\quad\text{ for each }t=0,\ldots,\horizon-1,\,\text{and }i \in \intbra{N},\\
            x_0^i = \xz^i, \text{ and } x^i_T =x^i_{d},\,x_{\horizon}^i \in \admfinst, \\
            x_t^i \in \Mbb,\,u_t^i \in \Ubb \quad\text{ for each }t=0,\ldots,\horizon-1,\,\text{and }i \in \intbra{N},
			\end{cases}
	\end{aligned}
\end{equation}
and for a fixed \(x_d\), the GW-cost is given by (recall the notation established in \S\ref{sec:GW_description})
\begin{align}\label{eq:GW_OCP_formulation}
\GW\bigl(x_{d},\xtarg\bigr) &\Let \inf_{\coupmat \in \probmeas\bigl(\tfrac{\bm{1}_{N}}{N},\tfrac{\bm{1}_{N}}{N}\bigr)} \dummysmallfunc \bigl(\costmatX_{i,i'},\costmatY_{j,j'}\bigr)\coupmat_{i,j}\coupmat_{i',j'} \nn \\ & = \inf_{\coupmat \in \probmeas\bigl(\tfrac{\bm{1}_{N}}{N},\tfrac{\bm{1}_{N}}{N}\bigr)} \inprod{\vectr(\coupmat)}{\dummybigfunc(x_{d},\xtarg)\vectr(\coupmat)},
\end{align}
where \(\dummysmallfunc(\cdot,\cdot)\) is a loss function that accounts for possible discrepancies between \(\costmatX_{i,i'} \Let \bigl(\costX(x_d^i,x_d^{i'})\bigr)_{i,i'} \), and \(\costmatY_{j,j'} \Let \bigl(\costY(\xtarg^j,\xtarg^{j'})\bigr)_{j,j'}\) and \(\dummybigfunc(x_{d},\xtarg) = \dummysmallfunc \bigl(\costmatX_{i,i'},\costmatY_{j,j'}\bigr)\) captures the \(N^2 \times N^2\) distance matrix as per the notation established in \S\ref{sec:GW_description}. Note that, for us \(\sourcevec = \tarvec = \frac{\bm{1}_{N}}{N}\) and thus \(\coupmat\) belongs to the set of all couplings \(\probmeas\bigl(\frac{\bm{1}_{N}}{N},\frac{\bm{1}_{N}}{N}\bigr)\). In the sequel, we will denote the optimizer of \eqref{eq:GW_OCP_formulation} by \(\optmat\). The optimal control problem (OCP) in \eqref{eq:optimal_navigator} can be understood as determining the destination for each agent such that the states are optimally guided to resemble the target shape \( x_\mathrm{target} \) with minimal control cost. 
\begin{remark}\label{rem:QAP_problem}
The OCP \eqref{eq:finite_horizon_OC} is a standard constrained quadratic program which, for a fixed \(\xz\) and \(x_d\), under the problem data \eqref{eq:data:I:system-st-con-dist}--\eqref{eq:data:III:system-dynamics}, admits a unique solution. When both \(\metspaceX\) and \(\metspaceY\) are \(d\)-dimensional Euclidean spaces, \((\xi,\mu)\mapsto g(\xi,\mu) \Let \frac{1}{2}\abs{\xi-\mu}^2\), and \(\costX(\cdot,\cdot)\) and \(\costY(\cdot,\cdot)\) are quadratic Euclidean functions, \(\dummybigfunc\) in \eqref{eq:GW_OCP_formulation} reduces to 
\begin{equation}
 \dummybigfunc(x_{d},\xtarg)_{ii'jj'} \Let \frac{1}{2}\abs{{\costX(x_d^{i},x_d^{i'})}-\costY(\xtarg^{j},\xtarg^{j'})}^2,
\end{equation}
and the corresponding optimization problem is known as a \emph{quadratic assignment} problem \cite{ref:MR:JK:JK:QAPGW-Neurips:2023}, which is still an NP-hard problem. Our formulation, based on the SDP-relaxation technique established in \cite{ref:JC:BTN:SHK:YSS:SDP-relaxation:2024} does not depend on a specific choice of \(g(\cdot,\cdot)\).
\end{remark}

\section{Main results}\label{sec:main_results}

The steering problem \eqref{eq:GW_OCP_formulation} is inherently NP-hard and challenging to solve numerically. To address this, we build on the approach in \cite{ref:JC:BTN:SHK:YSS:SDP-relaxation:2024} to develop a tight SDP relaxation for \eqref{eq:GW_OCP_formulation}, enabling globally optimal solutions in particular cases. 
\subsection{SDP relaxation of \eqref{eq:GW_OCP_formulation}}
Recall that for a given \(x_d\), the GW-problem is given by
\begin{align}
\GW\bigl(x_{d},\xtarg\bigr) \Let  \inf_{\coupmat \in \probmeas\left(\frac{\bm{1}_{N}}{N},\frac{\bm{1}_{N}}{N}\right)} \inprod{\vectr(\coupmat)}{\dummybigfunc(x_{d},\xtarg)\vectr(\coupmat)} \nn
\end{align}
where \(\dummybigfunc (x_d,\xtarg) \Let \dummysmallfunc\bigl( \costX(x_d^i,x_d^{i'}), \costY(\xtarg^j,\xtarg^{j'})\bigr)\) for each \(i, i', j, j' \in \intbra{N}\). By representing the quadratic term with a rank-one positive semi-definite matrix and reformulating the linear marginal equalities in vectorized form, we derive a relaxed SDP of the following structure:
\begin{equation}
	\label{eq:SDP_relaxation}
	\begin{aligned}
		&  \inf_{(\coupmat,\widehat{\coupmat}) \in \Rbb^{N \times N} \times \Rbb^{N^2 \times N^2} }  && \inprod{\dummybigfunc\bigl(x_{d},\xtarg\bigr)}{\widehat{\coupmat}}\\
		& \qquad\quad \,\,\sbjto && \begin{cases}
		\begin{pmatrix} \widehat{\coupmat} & \vectr(\coupmat) \\ \vectr(\coupmat)^{\top} & 1 \end{pmatrix} \succeq 0,\,\widehat{\coupmat} \succeq 0 \\
    \coupmat \in \probmeas \left( \frac{1}{N}\bm{1}_{N}, \frac{1}{N}\bm{1}_{N} \right), \\ 
    \widehat{\coupmat}\, \vectr(e_i \bm{1}^{\top}_N) = \frac{1}{N} \,\vectr(\coupmat), \\
    \widehat{\coupmat}\, \vectr(\bm{1}_{N} e_j^{\top}) = \frac{1}{N} \,\vectr(\coupmat)
		\end{cases}
	\end{aligned}
\end{equation}
One can interpret the relaxation \eqref{eq:SDP_relaxation} as a Lagrangian dual formulation of the GW problem \eqref{eq:GW_OCP_formulation}, where additional constraints are introduced to connect the linear and quadratic components of transportation plans. We refer the readers to \cite[\S 3]{ref:JC:BTN:SHK:YSS:SDP-relaxation:2024} for more details on the relaxation procedure and \cite[Appendix B]{ref:JC:BTN:SHK:YSS:SDP-relaxation:2024} for duality results.
\begin{proposition}\label{prop:SDP_existence}
 Consider the GW problem \eqref{eq:GW_OCP_formulation} and its SDP relaxation \eqref{eq:SDP_relaxation}. Then \eqref{eq:SDP_relaxation} admits a global optimal point over its feasible set.
\end{proposition}
\begin{proof}
For a fixed \(x_d\), denote the set of all feasible pairs \((\coupmat,\widehat{\coupmat}) \in \Rbb^{N \times N} \times \Rbb^{N^2 \times N^2}\)for the SDP by
\[
    \feasSet \Let \left\{(\coupmat,\widehat{\coupmat})\;\middle\vert\;  
    \begin{array}{@{}l@{}}
       \begin{pmatrix} \widehat{\coupmat} & \vectr(\coupmat) \\ \vectr(\coupmat)^{\top} & 1 \end{pmatrix}\succeq 0,\,\coupmat \in \probmeas \left(\frac{1}{N}\bm{1}_{N}, \frac{1}{N}\bm{1}_{N} \right),\,\coupmat \succeq 0\\ 
    \widehat{\coupmat}\, \vectr(e_i \bm{1}^{\top}_N) = \frac{1}{N} \,\vectr(\coupmat),\,
    \widehat{\coupmat}\, \vectr(\bm{1}_{N} e_j^{\top}) = \frac{1}{N} \,\vectr(\coupmat) 
        \end{array}
        \right\}.    
\]
Then, the SDP can be compactly written as 
\begin{equation}
	\label{eq:SDP_compact}
	\begin{aligned}
		& \inf_{(\coupmat,\widehat{\coupmat}) \in \Rbb^{N \times N} \times \Rbb^{N^2 \times N^2} }  && \inprod{\dummybigfunc\bigl(x_{d},\xtarg\bigr)}{\widehat{\coupmat}}\\
		& \qquad\quad \sbjto && \hspace{2mm}(\coupmat,\widehat{\coupmat}) \in \feasSet
	\end{aligned}
\end{equation}
Observe that: 
\begin{itemize}[leftmargin=*]

\item The set \(\probmeas \left(\frac{1}{N}\bm{1}_{N}, \frac{1}{N}\bm{1}_{N} \right)\) forms a convex and compact polytope.

\item The set of matrices \(\widehat{\coupmat} \in \Rbb^{N^2 \times N^2}\) that satisfy the linear constraints \(\widehat{\coupmat}\, \vectr(e_i \bm{1}^{\top}_N) = \frac{1}{N} \,\vectr(\coupmat)\) and \(\widehat{\coupmat}\, \vectr(\bm{1}_{N} e_j^{\top}) = \frac{1}{N} \,\vectr(\coupmat)\) have a trace of at most one. Thus the feasible set \(\feasSet\) consists of positive semi-definite matrices with trace at most one which is closed and bounded and thus compact. Moreover \(\feasSet\) is convex due to the linearity of trace. 
\item Finally, the map \((\coupmat,\widehat{\coupmat}) \mapsto \inprod{\dummybigfunc\bigl(x_{d},\xtarg\bigr)}{\widehat{\coupmat}}\) is continuous. 
\end{itemize}
Thus, solutions to \eqref{eq:SDP_relaxation} exist by Weierstrass's Theorem \cite[Theorem 2.2]{ref:OG10}.
\end{proof}
\begin{remark}[Approximation ratio and strong duality]\label{rem:on_SDP_relaxation}
The SDP relaxation \eqref{eq:SDP_relaxation} of the original GW problem \eqref{eq:GW_OCP_formulation} provides a framework for deriving a certificate of optimality, which can confirm whether a solution to the relaxed problem is also optimal for the original problem. To this end, let \(\optmat \in \probmeas\bigl(\frac{\bm{1}_{N}}{N}, \frac{\bm{1}_{N}}{N}\bigr)\) denote an optimizer of the original GW problem \eqref{eq:GW_OCP_formulation}, and let \((\sdpoptmat, \sdpoptbigmat) \in \feasSet\) represent the optimizer of the relaxed problem \eqref{eq:SDP_relaxation}. By construction, \(\sdpoptmat \in \probmeas\bigl(\frac{\bm{1}_{N}}{N}, \frac{\bm{1}_{N}}{N}\bigr)\), indicating that \((\sdpoptmat, \sdpoptbigmat)\) is feasible for \eqref{eq:GW_OCP_formulation} as well. Consequently, we obtain the inequality \(\inprod{\sdpoptbigmat}{\dummybigfunc\bigl(x_{d},\xtarg\bigr)} \le \inprod{\vectr(\sdpoptmat)}{\dummybigfunc\bigl(x_{d},\xtarg\bigr) \vectr(\sdpoptmat)}\). This implies that if the ratio
\begin{align}\label{eq:approx_ration_condition}
\frac{\inprod{\vectr(\sdpoptmat)}{\dummybigfunc\bigl(x_{d},\xtarg\bigr)\vectr(\sdpoptmat)}}{\inprod{\sdpoptbigmat}{\dummybigfunc}} = 1,
\end{align}
then \(\sdpoptmat\) is a \emph{globally optimal} solution to \eqref{eq:GW_OCP_formulation} \cite{ref:JC:BTN:SHK:YSS:SDP-relaxation:2024}. In \S \ref{sec:numerics}, we will observe that in our numerical experiments, this ratio remains very close to one. Moreover, one can obtain \emph{exact solution} if the optimizers \((\sdpoptmat, \sdpoptbigmat)\) satisfies the rank condition: 
\begin{align}\label{eq:rank_condition}
   \mathrm{rank} \begin{pmatrix}
        \sdpoptbigmat & \vectr(\sdpoptmat) \\ \vectr(\sdpoptmat)^{\top} & 1
    \end{pmatrix} =1 .
\end{align}
The condition \eqref{eq:rank_condition} guarantees that the strong duality holds, i.e., duality gap between the SDP \eqref{eq:SDP_relaxation} and its dual is zero; see \cite[Proposition B.3]{ref:JC:BTN:SHK:YSS:SDP-relaxation:2024}. We will also demonstrate in our numerical example that the condition \eqref{eq:rank_condition} holds. 
\end{remark}
Using the SDP relaxation \eqref{eq:SDP_relaxation}, we aim to solve the relaxed optimal control problem (OCP):
\begin{equation}\label{eq:SDP_relaxed_optimal_navigator}
	\begin{aligned}
		& \inf_{x_{d} \in \admfinst} \,\, J(x_d)\Let \rcost(\xz,x_{d}) + \GWSDP\lr{x_{d},\xtarg},
	\end{aligned}
\end{equation}
where \(\GWSDP\lr{x_{d},\xtarg}\) corresponds to the relaxation \eqref{eq:SDP_relaxation}.
\begin{theorem}
    Consider the OCP \eqref{eq:SDP_relaxed_optimal_navigator} with the problem data \eqref{eq:data:I:system-st-con-dist}--\eqref{eq:data:III:system-dynamics} and assume that it is strictly feasible. Then, \eqref{eq:SDP_relaxed_optimal_navigator} admits a globally optimal solution over its feasible set. 
\end{theorem}
\begin{proof}
Note that, \(x_d\) is constrained to the compact set \(\admfinst\) and observe that: 
\begin{itemize}[leftmargin=*]
    \item For each fixed \(x_d\), in the OCP \eqref{eq:finite_horizon_OC}, the map \(\mu \mapsto \inprod{\mu}{R \mu}\) is convex and the constraint sets \(\Mbb, \Ubb,\) and \(\admfinst\) are all convex sets, and the Slater condition hold. 
    \item For each fixed \(x_d\), in \eqref{eq:SDP_relaxation}, the set \(\feasSet\) is convex; see Proposition \ref{prop:SDP_existence}, and the function \(\widehat{\coupmat} \mapsto \inprod{\dummybigfunc(x_{d},\xtarg)}{\widehat{\coupmat}}\) is convex. Moreover, \eqref{eq:SDP_relaxation} is strictly feasible. 
\end{itemize}
Thus, the maps \(x_d \mapsto \rcost(\xz, x_d,)\) and \(x_d \mapsto \GWSDP(x_d,\xtarg)\) are lower semicontinuous \cite[Lemma 5.7]{ref:still2018lectures}. By a version of Weierstrass's Theorem, the existence of a solution follows \cite[Box 1.3, pp.9--10]{ref:santambrogio2023course}.\end{proof} 
\begin{remark}
The choice of an algorithm to solve \eqref{eq:SDP_relaxed_optimal_navigator} depends on the smoothness and convexity properties of the objective function. In the absence of additional assumptions on the problem data, \(J(\cdot)\) is generally nonsmooth and nonconvex with respect to the decision variable. Several algorithms are available for tackling such problems, including interior point methods \cite{ref:karmista_interior_nonsmooth, ref:schmidt2015interior}, proximal-like methods \cite{ref:grimmer_nonsmooth-nonconvex-opt}, and projective penalty methods \cite{ref:nonsmooth-projective-norkin2024exact}. In fact, with \(\admfinst\) convex, the OCP \eqref{eq:SDP_relaxed_optimal_navigator} admits the same form given in \cite[\S 3]{ref:nonsmooth-projective-norkin2024exact}, and the penalty method established therein is exact \cite[Theorem 3]{ref:nonsmooth-projective-norkin2024exact}, which can be utilized to solve \eqref{eq:SDP_relaxed_optimal_navigator}. These approaches ensure that the sequence of iterates converges to a stationary point, which may correspond to a KKT point or a Fritz-John point depending on the constraint qualification conditions. A comprehensive analysis of the structural and regularity properties of \(J(\cdot)\), along with the design of tailored algorithms to solve \eqref{eq:SDP_relaxed_optimal_navigator}, will be the subject of future work. 
\end{remark}

\section{Numerical experiments}\label{sec:numerics}
We consider a constrained formation shape control problem with \(10\) agents which follows the linear dynamics
\begin{align}\label{eq:num:dynamics}
    x_{t+1}^{i} = \begin{pmatrix}0.5 & 0.2 \\ 0.1 & 0.4 \end{pmatrix} x_t^{i} + \begin{pmatrix} 1 & 0 \\ 0 & 1
    \end{pmatrix} u_t^{i} \quad\text{for }i\in \intbra{10}\text{ and }t \in \N.
\end{align}
The constraint sets are given by \(\Mbb \Let \aset[]{\xi \suchthat \norm{\xi}_{\infty} \le 20}\), \(\Ubb \Let \aset[]{\mu \suchthat \norm{\mu}_{\infty} \le 20}\) and \(\admfinst \Let \aset[]{\eta \suchthat \norm{\eta}_{\infty} \le 20}\). Fix \(\horizon = 10\) and consider the OCP:
\begin{equation}
	\label{eq:num:finite_horizon_OCP}
	\begin{aligned}
		& \inf_{x_d} \inf_{(u_t)_{t=0}^{\horizon-1}}  && \sum_{i=0}^{N} \sum_{t=0}^{\horizon-1} \eps \norm{u_t^i}^2_{I_2} + \GW\bigl(x_{d},\xtarg\bigr)\\
		& \quad \sbjto && \begin{cases}
			\text{dynamics }\eqref{eq:num:dynamics},\,x^i_t \in \Mbb, \, u^i_t \in \Ubb\text{ for each }i \in \intbra{10},\,t \in \aset[]{0,\ldots,\horizon-1}, \\
            \text{and }x_d \in \admfinst
		\end{cases}
	\end{aligned}
\end{equation}
where \(I_2 \in \Rbb^{2 \times 2}\) is the identity matrix, and \(\eps>0\) is a parameter that we employ to strike a balance between the control cost and GW distance. In the GW cost, we employed \(\dummysmallfunc(x,y)\Let \frac{1}{2}\abs{x-y}^2\).  

As the agents' initial states, we randomly selected ten points from the set \( \{(x,y) \in \lcrc{-15}{-12} \times \lcrc{-2}{2}\}\) and fixed a target formation given in Figure \ref{fig:comparison} (see the left-hand subfigure in the first row) which our agents need to mimic at the final time. We employed Julia v1.11.1 \cite{ref:Julia:bezanson2017} and its \texttt{JuMP} library to perform all our numerical computations and optimization tasks.
\begin{wrapfigure}{r}{5.5cm}
\caption{Approximation ratio}\label{fig:approx_ratio}
\includegraphics[scale=0.24]{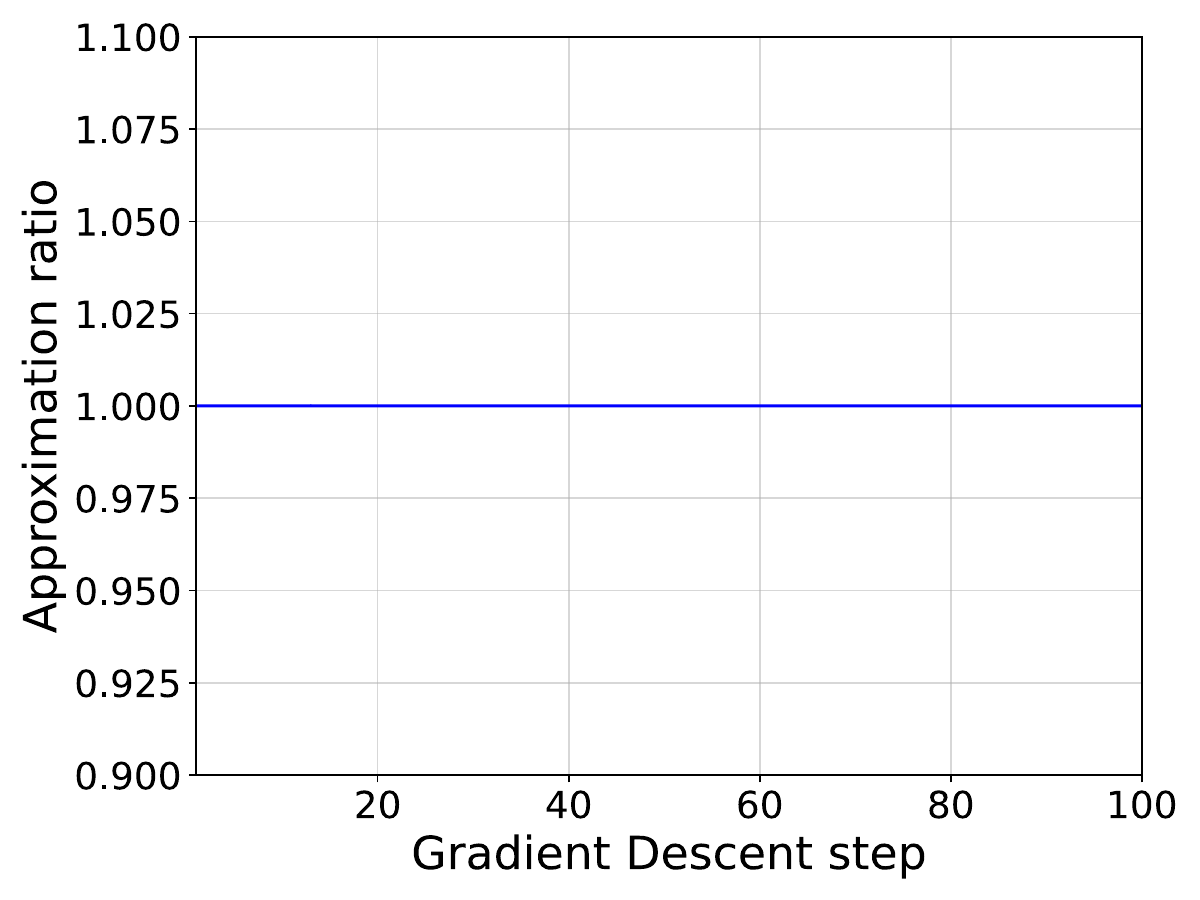}
\end{wrapfigure}
To convert the GW distance in \eqref{eq:GW_OCP_formulation} we performed the relaxation given in \eqref{eq:SDP_relaxation} and the ensuing SDP was solved using MOSEK \cite{ref:mosek}. To solve the intermediate OCP \eqref{eq:finite_horizon_OC}, which is a convex quadratic program, we employed MOSEK as well. Finally, we employed the interior-point solver IPOPT \cite{ref:IPOPTwachter2006implementation} to solve the problem \eqref{eq:SDP_relaxed_optimal_navigator}. The tolerance of MOSEK for solving both the problems \eqref{eq:finite_horizon_OC} and \eqref{eq:SDP_relaxation} was set at \(10^{-5}\); the tolerance for the interior-point algorithm was set at \(10^{-6}\). 

We calculated the approximation ratio for \(10\) different runs using the optimizers \((\sdpoptmat,\sdpoptbigmat)\) of the SDP to check the accuracy of the solution and we found that, in average
\(
\tfrac{\inprod{\vectr(\sdpoptmat)}{\dummybigfunc(x_{d},\xtarg)\vectr(\sdpoptmat)}}{\inprod{\sdpoptbigmat}{\dummybigfunc(x_{d},\xtarg)}} \approx 1,\nn
\)
which indicates that \(\sdpoptbigmat\) is a globally optimal solution to \eqref{eq:GW_OCP_formulation} and our method offers a high-quality approximation of the optimal transport solution for the GW distance; see Figure \ref{fig:approx_ratio}. Consequently, the final configuration of the agents nearly mirrors the target formation; see Figure \eqref{fig:comparison} (the right-hand subfigure in the first row, and the subfigure in the second row). Moreover, we observed that the optimizers \((\sdpoptmat, \sdpoptbigmat)\) satisfies the rank condition \eqref{eq:rank_condition} 
which indicates that strong duality holds for this example. As noted in \cite{ref:JC:BTN:SHK:YSS:SDP-relaxation:2024}, while the SDP is solvable in polynomial time, computational efficiency could be further improved, for example, using sparsification of the SDP solutions. We will explore this in our future investigations. 
\begin{figure}[ht]
\centering
\begin{minipage}{0.4\linewidth}
\includegraphics[scale=0.2]{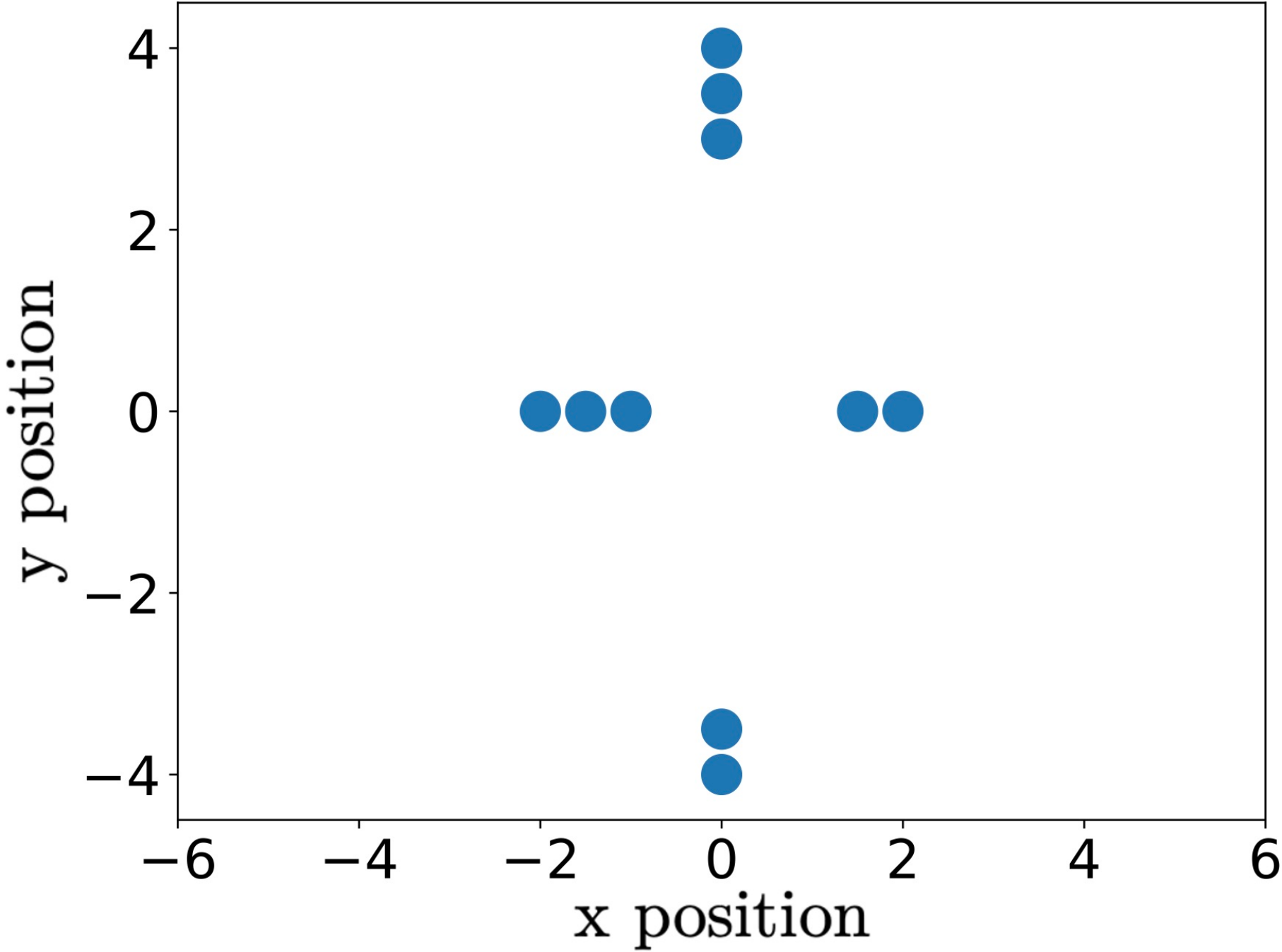}
\label{fig:target-label}
\end{minipage}%
\hspace{5mm}
\begin{minipage}{0.4\linewidth}
\includegraphics[scale=0.21]{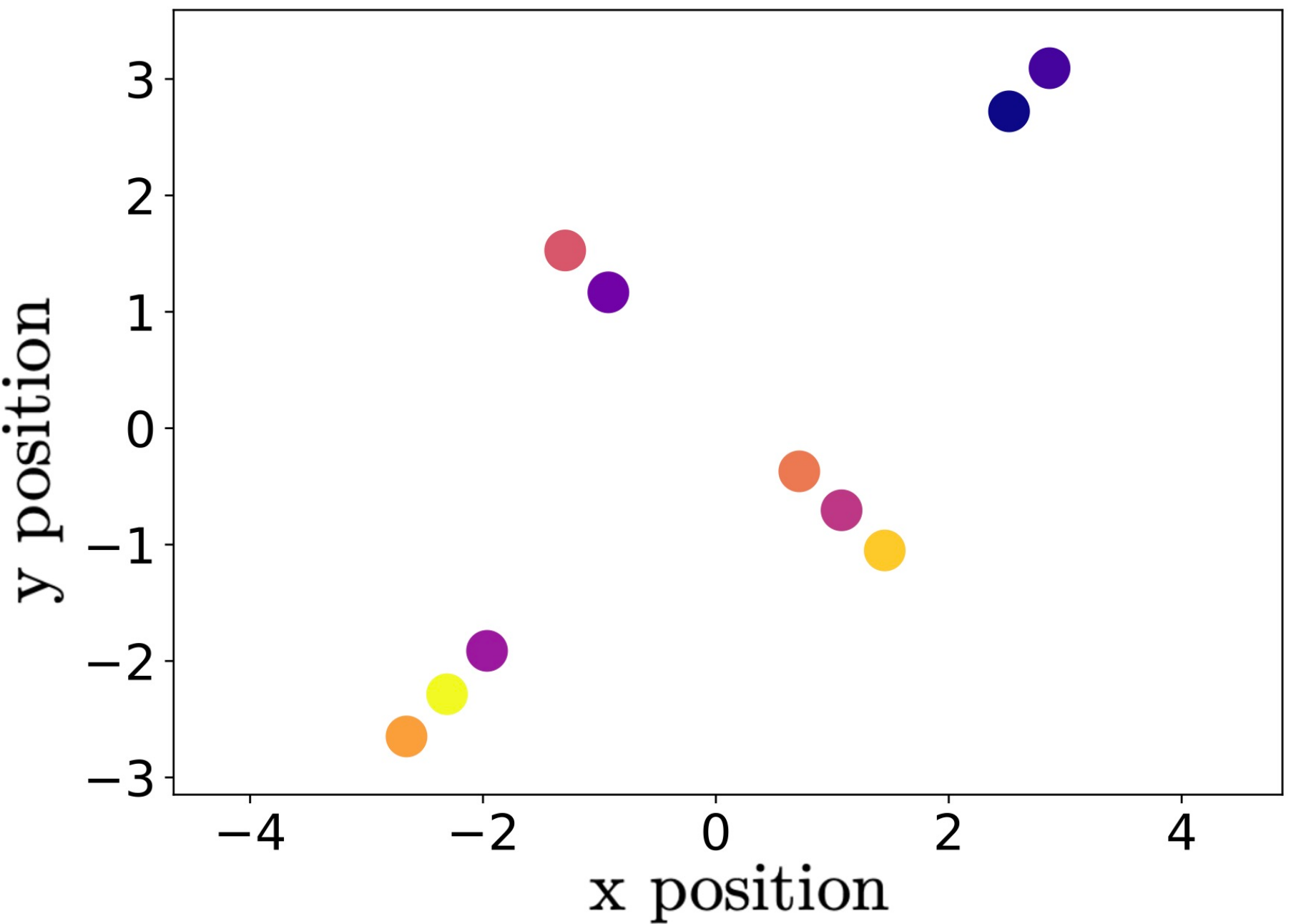}
\label{fig:terminal-label}
\end{minipage}
\vspace{-2mm}
\begin{minipage}{0.98\linewidth}
\centering
\includegraphics[scale=0.27]{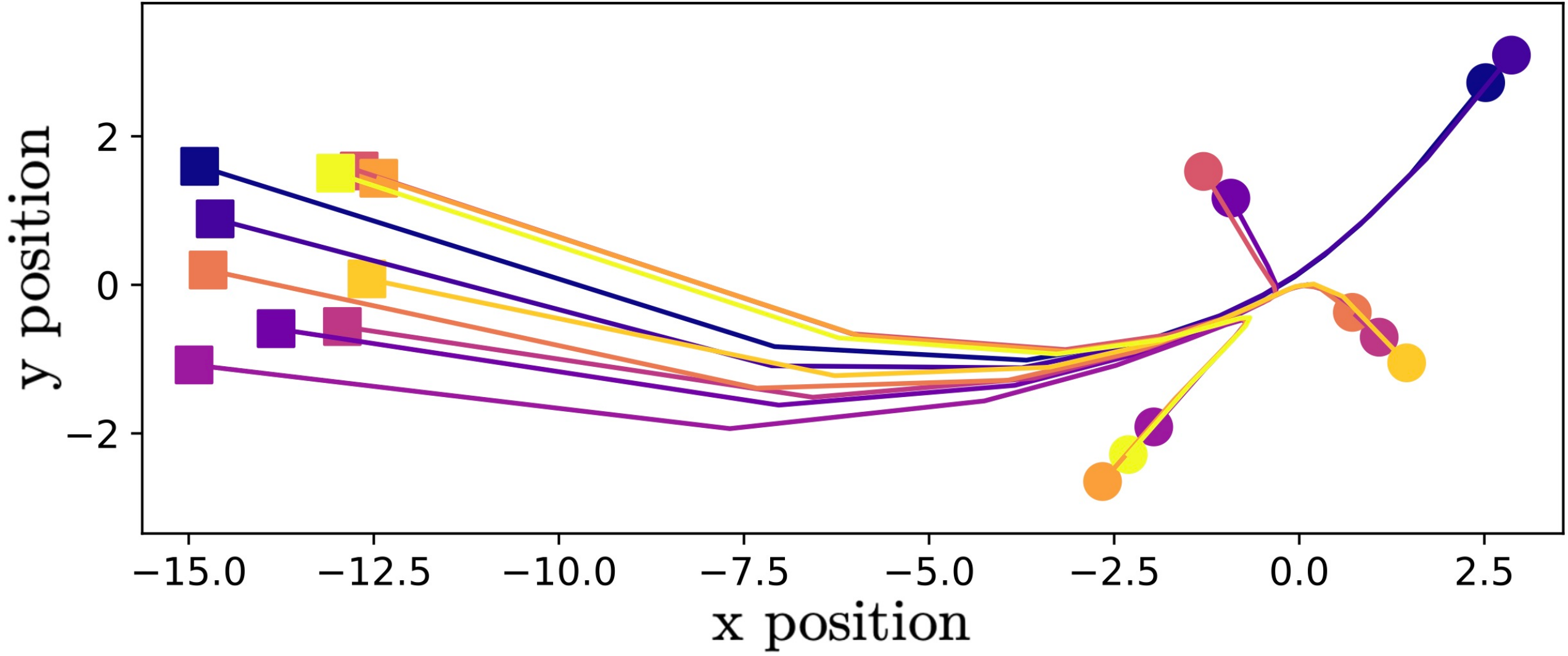}
\end{minipage}
\caption{In the top row, the left-hand and right-hand subfigures depict the target shape and the final shape, respectively. The bottom row depicts the randomly chosen initial points (squares) and their trajectories up to the final formation (circles), with \(\eps = 0.5\).}
\label{fig:comparison}
\end{figure}
We also record the control cost and the cost associated with solving the SDP problem in Table \eqref{tab:optimization_stats} for some \(\eps\) in \eqref{eq:num:finite_horizon_OCP}. Table \ref{tab:optimization_stats} highlights a trade-off between the overall control cost and the GW distance between the target and final formation. 

The previous numerical illustration presented a case where both \(\costX(\cdot,\cdot)\) and \(\costY(\cdot,\cdot)\) were standard Euclidean distances. However, the flexibility of the GW distance allows us to adopt different underlying metrics. For example, let us consider a scenario in which 6 agents are to be separated into 2 groups in the final state, each consisting of 3 agents (the numbers can be adjusted as long as they are meaningful). To this end, we fixed fully connected graph structure as given in Figure \ref{fig:graph_formation_2} (left-hand subfigure) where we assigned weights to all pairs of edges in the graph in the following manner:
\begin{align}
\label{eq:graph_distance}
d_{\mathrm{graph}}(i,i^{'}) \Let
\begin{cases} 
    2  &  \text{if } \{i,i'\} \subset \{1,2,3\} \,\,or\,\, \{i,i'\} \subset \{4,5,6\},\\
    4  & \text{otherwise.}  \\
\end{cases}
\end{align}
To achieve the desired grouping, we assign smaller weights to edges connecting nodes within the same group, and larger weights to edges that span across groups.  Note that \(d_{\mathrm{graph}}\) is a metric. With this premise, we considered the following GW distance:
\begin{align}
\label{eq:graph_GW}
\GW_{C}\bigl(x_{d}) \Let  \inf_{\coupmat \in \probmeas\left(\frac{\bm{1}_{N}}{N},\frac{\bm{1}_{N}}{N}\right)} \inprod{\vectr(\coupmat)}{\dummybigfunc_{C}(x_{d})\vectr(\coupmat},
\end{align}
where \(\dummybigfunc_{C} (x_d) \Let \dummysmallfunc\bigl( d_{\mathrm{graph}}(i,i^{'}), \costY(x_{d}^j,x_{d}^{j'})\bigr)\) for each \(i,i',j,j' \in \intbra{N}\); and \(g(\cdot,\cdot)\) is same as before. Finally, to achieve separation of agents, we consider the following problem. 
\begin{wraptable}{l}{60mm}
\begin{tblr}{l c c }
\hline[2pt]
 \SetRow{azure9}
        \(\eps\) & Control cost  & GW distance \\
        \hline
        0.1 & 59.82 & 0.02 \\
        \hline
        0.5 & 51.63 & 0.50 \\
        \hline
        1.0 & 48.12 & 2.51 \\
        \hline
        2.5 & 40.60 & 15.21 \\
 \hline[2pt]
\end{tblr}
\centering
\vspace{2.5mm}
\caption{Trade-off between control cost and GW distance between agents' state and target formation, with 10 agents. Observe that increasing \(\eps\) leads to a lower control cost; however, this comes at the expense of a higher GW cost, resulting in a misaligned formation.}
\label{tab:optimization_stats}
\end{wraptable}
\begin{equation}\label{eq:Problem_with_graph}
	\begin{aligned}
		& \inf_{x_{d} \in \admfinst} \,\, \rcost(\xz,x_{d}) + \GW_{C}\bigl(x_{d}\bigr),
	\end{aligned}
\end{equation}
where \( \rcost(\xz,x_{d})\) takes care of the control cost while \(\GW_{C}(x_{d})\) accounts for the separation of agents. Consequently, we employed the same SDP relaxation strategy given in \S\ref{sec:main_results} for \(\GW_{C}\) and solved \eqref{eq:Problem_with_graph} with exactly the same problem data and solver specifications. 
\begin{figure}[ht]
\centering
\begin{minipage}{0.45\linewidth}
\hspace{-10mm}
\includegraphics[scale=0.28]{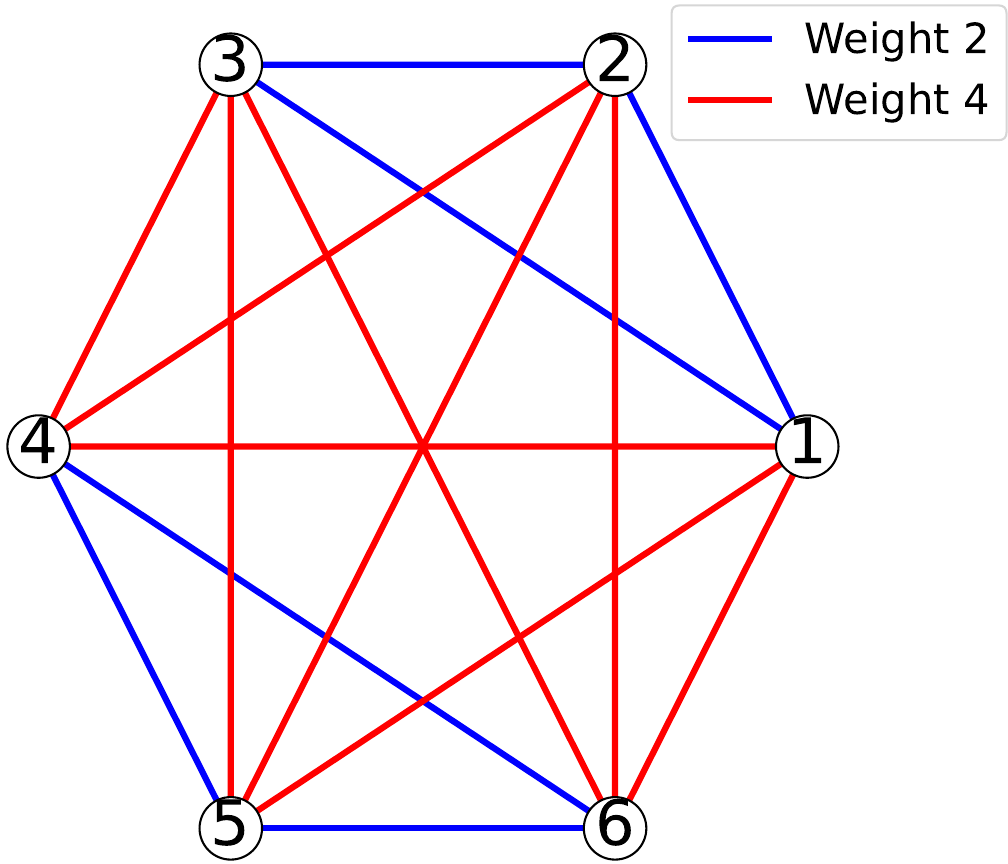}
\end{minipage}%
\begin{minipage}{0.45\linewidth}
 \hspace{-28mm}
\includegraphics[scale=0.24]{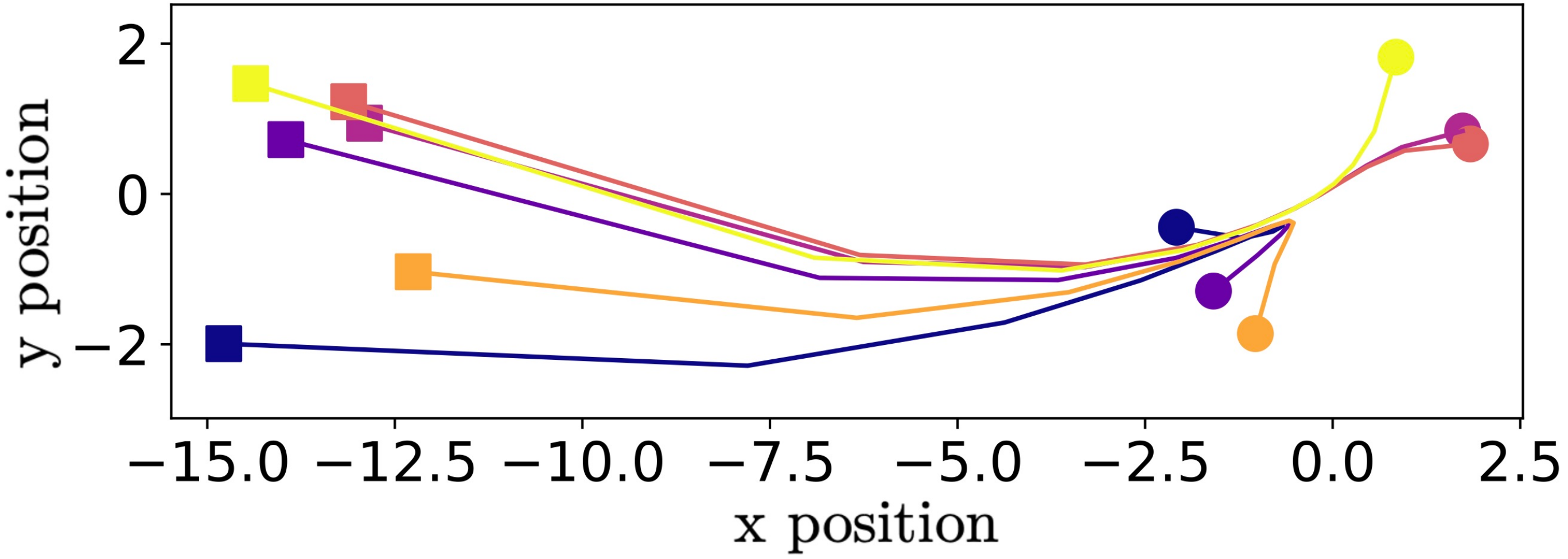}
\end{minipage}%
\caption{The left-hand subfigure depicts the target graph structure considered in \eqref{eq:graph_distance}, the right-hand subfigures shows the trajectory achieved, transitioning from squares to circles, with \(\eps = 0.5\). Note that in the final shape, the agents are successfully separated into 2 groups, respecting the graph structure. This is ensured by the term \(\GW_c(x_d)\).}
\label{fig:graph_formation_2}
\end{figure}

\section{Discussion and future work}\label{sec:conclusion}
We developed an SDP-relaxation approach for multi-agent shape control, guided by the Gromov-Wasserstein (GW) distance, to achieve a target formation within a constrained, minimum-energy optimal control framework. Drawing on insights from \cite{ref:JC:BTN:SHK:YSS:SDP-relaxation:2024}, our algorithm provides an optimality certificate, ensuring the exactness and global optimality of the SDP solutions. While the NP-hardness associated with the original GW problem is mitigated by this relaxation, several challenges remain, opening up various avenues for theoretical and numerical enhancements:
\begin{itemize}[leftmargin=*]
\item Enhancing the computational efficiency of our algorithm by introducing sparsification techniques to the SDP solutions. By selectively reducing the complexity of the SDP structure while retaining essential features, it may be possible to lower the computational burden without compromising accuracy.  

\item Incorporating entropy regularization into our framework, combined with sparsification-driven Sinkhorn-like algorithms, to address an approximate version of the original Gromov-Wasserstein (GW) problem. This approach could leverage the computational advantages of entropy-regularized optimal transport while maintaining a balance between precision and efficiency, and consequently could enable faster and more scalable solutions.

\item Investigating scenarios where external disturbances influence the system and designing a robust model predictive control (MPC) algorithm to stabilize the system while maintaining the desired formation. In such cases, the use of an unbalanced GW cost \cite{ref:GWML:unbalanced_GW} would be more appropriate, as it allows for greater flexibility in handling mismatched distributions and system perturbations. This extension could enable the algorithm to adapt to practical challenges and ensure reliable performance in real-world applications.
\end{itemize}

\acks{This work was supported by the JSPS KAKENHI under Grant Number JP21H04875.}

\bibliography{refs}

\end{document}